\documentclass{article}[11pt]
\usepackage{amsthm,amsmath}
\usepackage{amssymb}
\input xypic
\usepackage[all,tips]{xy}
\newtheorem{theo}[subsection]{Theorem}
\newtheorem{lem}[subsection]{Lemma}
\newtheorem{prop}[subsection]{Proposition}
\newtheorem{corr}[subsection]{Corollary}
\newtheorem{exam}[subsection]{Example}
\newtheorem{rem}[subsection]{Remark}

\newcommand{\cC}{{\cal C}}
\newcommand{\cZ}{{\cal Z}}

\newcommand{\F}{{\bf F}}

\newcommand{\Vect}{{\cal V}{\it ect}}
\newcommand{\YD}{{\cal YD}}
\newcommand{\da}{\mbox{-}}
\newcommand{\se}{\mbox{:}}
\newcommand{\End}{{\cal E}{\it nd}}

\newcommand{\eq}{\mbox{\Large =}}

\begin{document}
\author{Alexei Davydov}
\title{Full centre of an $H$-module algebra}
\maketitle
\date{}
\begin{center}
Max Planck Institut f\"ur Mathematik\\
Vivatsgrasse 7,
53111 Bonn, Germany
\end{center}
\begin{center}
davydov@mpim-bonn.mpg.de
\end{center}
\begin{abstract}
We apply the full centre construction, defined in \cite{da}, to algebras in and module categories over categories of representations of Hopf algebras.
\newline
We obtain a compact formula for the full centre of a module algebra over a Hopf algebra. 
\end{abstract}
\tableofcontents
\section{Introduction}

Motivated by algebraic structures appearing in Rational Conformal Field Theory (see \cite{kr}) we studied in \cite{da} a construction associating to an algebra in a monoidal category a commutative algebra ({\em full centre}) in the monoidal centre of the monoidal category. In loc.cit. we established Morita invariance of this construction by extending it to module categories. 

In this note we apply the full centre construction to algebras in monoidal categories $H\da Mod$ of modules over a Hopf algebra $H$ ($H$-module algebras for short). 
It is known (see section \ref{moncen} and references therein) that the monoidal centre of the category of modules over a Hopf algebra $H$ is equivalent to the category of Yetter-Drinfeld modules over $H$.
The main result of this note (theorem \ref{main}) gives a closed formula for the full centre of an $H$-algebra:
\newline
{\bf Theorem (Main).}
Let $A$ be an $H$-module algebra.
\newline
Then its full centre $Z(A)$ of an $H$-module algebra $A$ coincides with the centraliser $C_{A\#H}(A)$. The Yetter-Drinfeld structure ($H$-action and $H$-coaction) on $C_{A\#H}(A)$ has the form:
$$h(\sum_ia_i\# g_i) = \sum_{(h),i}h_{(2)}(a_i)\#S^2(h_{(3)})g_iS(h_{(1)}),$$ $$\psi(\sum_ia_i\# g_i) = \sum_{i,(g_i)}S^{-1}((g_i)_{(2)})\otimes a_i\#(g_i)_{(1)},$$ where $\sum_ia_i\# g_i\in C_{A\#H}(A)$.
\newline
For the definition of the smash product see section \ref{halg}.

As an immediate corollary of the properties of the full centre established in \cite{da} we have the following (corollary \ref{cor}):
\newline
The centraliser $C_{A\#H}(A)$ is a commutative algebra in the braided category $\YD(H)$ of Yetter-Drinfeld modules over $H$.
\newline
Moreover, the centraliser $C_{A\#H}(A)$ is invariant with respect to Morita equivalences in $H\da Mod$, i.e. if $A$ and $B$ are two $H$-algebras and $M$ is a $B$-$A$-bimodule, equipped with compatible $H$-action, which induce an equivalence between categories of modules $M$$\otimes_A\da :{_{A}(H\da Mod)}\to{_{B}(H\da Mod)}$ then $C_{A\#H}(A)$ and $C_{B\#H}(B)$ are isomorphic as algebras in $\YD(H)$.

In particular, by examining the case $Z(k)$ we recover a result from \cite{cvz} that the opposite algebra $H^{op}$ of a Hopf algebra is a commutative algebra in the category of Yetter-Drinfeld modules for $H$ if we equip it with the adjoint $H$-action 
and $H$-coaction given by the coproduct. 

We conclude by applying our main theorem to the cases of group Hopf algebra and its dual. 

Throughout the paper we freely use definitions and results from \cite{da}.

\section{String diagrams conventions}

All Hopf algebras will be assumed to have invertible antipodes. Although we will frequently use Sweedler's notations for comultiplication and coaction, most of our computations will be done using string diagrammatic presentation. We read string diagrams from top to bottom. The string diagrams for multiplication, comultiplication, unit, counit, and antipode are: 

\bigskip
$\xygraph{ !{0;/r1.0pc/:;/u1.0pc/::}
*{} :@{-}[dr] *{\circ}
(
:@{-}[ru]
,
:@{-}[d]
}
\quad\quad\quad\quad\quad
\xygraph{ !{0;/r1.0pc/:;/u1.0pc/::}
*{} :@{-}[d] *{\circ}
(
:@{-}[ld]
,
:@{-}[rd]
}
\quad\quad\quad\quad\quad\quad
\xygraph{ !{0;/r1.0pc/:;/u1.0pc/::}
*{} :@{}[d] *{\circ}
:@{-}[d]
}
\quad\quad\quad\quad\quad\quad
\xygraph{ !{0;/r1.0pc/:;/u1.0pc/::}
*{} :@{-}[d] *{\circ}
:@{}[d]
}
\quad\quad\quad\quad\quad\quad
\xygraph{ !{0;/r1.0pc/:;/u1.0pc/::}
*{} :@{-}[d] *{\circ}
:@{-}[d]
}$
\bigskip
\bigskip\newline
respectively. 
The fact that the comultiplication is a homomorphism of algebras (or the fact that multiplication is a homomorphism of coalgebras) has the following diagrammatic form:
$$
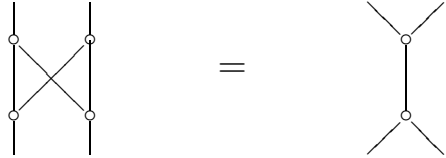
\begin{xy} (-25,0)*+!R{\xygraph{ !{0;/r1.2pc/:;/u1.2pc/::}
*{} :@{-}[d] *{\circ}
(
:@{-}[dd] *{\circ}
 (
 :@{-}[d]
 ,
 :@{-}[uurr]="r" *{\circ}
 :@{-}[u]
 )
,
:@{-}[ddrr] *{\circ}
 (
 :@{-}"r"
 ,
 :@{-}[d]
 )
}
}
\POS(3,-4)*{{\eq}}
\POS(20,0)*+!L{\xygraph{ !{0;/r1.2pc/:;/u1.2pc/::}
*{} :@{-}[dr] *{\circ}
(
:@{-}[ur]
,
:@{-}[dd] *{\circ}
 (
 :@{-}[dl]
 ,
 :@{-}[dr]
 )
)
}
}
\end{xy}
$$
\bigskip\bigskip
\newline
The antipode axioms have the form:
\bigskip
$$
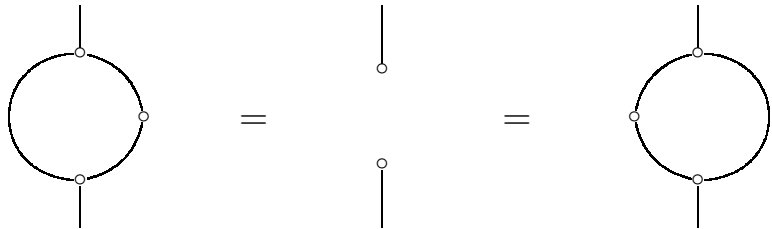
\begin{xy} (-45,0)*+!R{\xygraph{ !{0;/r2.0pc/:;/u2.0pc/::}
*{} :@{-}[d(.75)] *{\circ}
(
:@{-}@/^6pt/[dr] *{\circ}
:@{-}@/^6pt/[dl] *{\circ}="b"
:@{-}[d(.75)]
,
:@{-}@(l,l)"b"  
)
}
}
\POS(-23,-7)*{{\eq}}
\POS(-7,0)*+!L{\xygraph{ !{0;/r2.0pc/:;/u2.0pc/::}
*{} :@{-}[d] *{\circ}
:@{}[d(1.5)] *{\circ}
:@{-}[d]
}
}
\POS(12,-7)*{{\eq}}
\POS(35,0)*+!L{\xygraph{ !{0;/r2.0pc/:;/u2.0pc/::}
*{} :@{-}[d(.75)] *{\circ}
(
:@{-}@/_6pt/[dl] *{\circ}
:@{-}@/_6pt/[dr] *{\circ}="b"
:@{-}[d(.75)]
,
:@{-}@(r,r)"b"  
)
}
}
\end{xy}
$$
\bigskip
\bigskip
\bigskip

The proof of the following auxiliary statement is an example of string diagram arguments.
\begin{lem}\label{aux}
The following identity holds for any element $g,h$ of a Hopf algebra $H$:
$$\sum (gS(h_{(1)}))_{(1)}h_{(2)}\otimes(gS(h_{(1)}))_{(2)} = \sum_{(g)}g_{(1)}\otimes g_{(2)}S(h).$$
\end{lem}
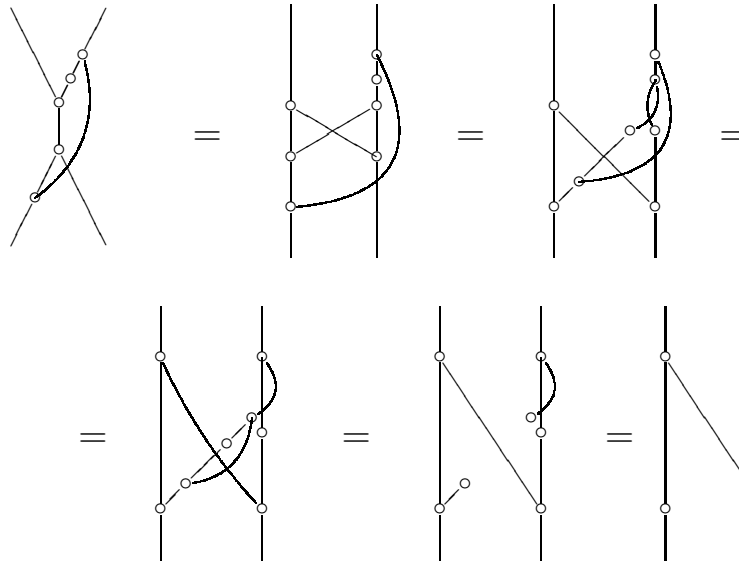
\begin{proof}

$$\begin{xy} (-135,0)*+!R{
\xygraph{ !{0;/r1.5pc/:;/u1.5pc/::}
*{} :@{-}[ddr] *{\circ}
(
:@{-}[d] *{\circ}
 (
 :@{-}[d(1)l(.5)]="l" *{\circ}
 :@{-}[d(1)l(.5)]
 ,
 :@{-}[ddr]
 )
,
:@{-}[u(.5)r(.25)] *{\circ}
:@{-}[u(.5)r(.25)] *{\circ}
 (
 :@{-}[u(1)r(.5)]
 ,
 :@{-}@/^10pt/"l"
 )
)
}
}
\POS(-110,-11)*{\eq}
\POS(-100,0)*+!L{
\xygraph{ !{0;/r1.6pc/:;/u1.6pc/::}
*{} :@{-}[dd] *{\circ}
(
:@{-}[d] *{\circ}
 (
 :@{-}[r(1.7)u] *{\circ}
  (
  :@{-}[u(.5)] *{\circ}
  :@{-}[u(.5)]="u" *{\circ}
  :@{-}[u]
  ,
  :@{-}[d]="m" *{\circ}
  :@{-}[dd]
  )
 ,
 :@{-}[d] *{\circ}
  (
  :@{-}[d]
  ,
  :@{-}@/_25pt/"u"
  )
 )
,
:@{-}"m"
)
}
}
\POS(-75,-11)*{\eq}
\POS(-65,0)*+!L{
\xygraph{ !{0;/r1.6pc/:;/u1.6pc/::}
*{} :@{-}[dd] *{\circ}
(
:@{-}[dd] *{\circ}
 (
 :@{-}[d]
 ,
 :@{-}[u(.5)r(.5)]="l" *{\circ}
 :@{-}[u(1)r(1)] *{\circ}
 :@{-}@/_5pt/[r(.5)u(1)]="u" *{\circ}
 :@{-}[u(.5)] *{\circ}
  (
  :@{-}[u]
  ,
  :@{-}@/^20pt/"l"
  )
 )
,
:@{-}[ddrr] *{\circ}
 (
 :@{-}[d]
 ,
 :@{-}[u(1.5)] *{\circ}
 :@{-}@/^3pt/"u"
 )
)
}
}
\POS(-40,-11)*{\eq}
\end{xy}
$$
\bigskip
\bigskip
\bigskip
\bigskip

$$\begin{xy} 
\POS(-145,-11)*{\eq}
\POS(-135,0)*+!R{
\xygraph{ !{0;/r1.6pc/:;/u1.6pc/::}
*{} :@{-}[d] *{\circ}
(
:@{-}[ddd] *{\circ} 
 (
 :@{-}[d]
 ,
 :@{-}[u(.5)r(.5)] *{\circ}
  (
  :@{-}[u(.8)r(.8)] *{\circ}
  :@{-}[u(.5)r(.5)]="m" *{\circ}
  :@{-}@/_7pt/[r(.2)u(1.2)]="u" *{\circ}
  :@{-}[u]
  ,
  :@{-}@/_7pt/"m"
  )
 )
,
:@{-}@/_3pt/[dddrr] *{\circ}
 (
 :@{-}[d]
 ,
 :@{-}[u(1.5)] *{\circ}
 :@{-}"u"
 )
)
}
}
\POS(-110,-11)*{\eq}
\POS(-100,0)*+!L{
\xygraph{ !{0;/r1.6pc/:;/u1.6pc/::}
*{} :@{-}[d] *{\circ}
(
:@{-}[ddd] *{\circ} 
 (
 :@{-}[d]
 ,
 :@{-}[u(.5)r(.5)] *{\circ}
  (
  :@{}[u(.8)r(.8)] 
  :@{}[u(.5)r(.5)]="m" *{\circ}
  :@{-}@/_7pt/[r(.2)u(1.2)]="u" *{\circ}
  :@{-}[u]
  ,
  :@{}"m"
  )
 )
,
:@{-}[dddrr] *{\circ}
 (
 :@
{-}[d]
 ,
 :@{-}[u(1.5)] *{\circ}
 :@{-}"u"
 )
}
}
\POS(-75,-11)*{\eq}
\POS(-70,0)*+!L{
\xygraph{ !{0;/r1.6pc/:;/u1.6pc/::}
*{} :@{-}[d] *{\circ}
(
:@{-}[ddd] *{\circ} 
 (
 :@{-}[d]
 ,
 :@{}[u(.5)r(.5)] 
  (
  :@{}[u(.8)r(.8)] 
  :@{}[u(.5)r(.5)]="m" 
  :@{}[r(.2)u(1.2)]="u" *{\circ}
  :@{-}[u]
  ,
  :@{}"m"
  )
 )
,
:@{-}[dddrr] *{\circ}
 (
 :@
{-}[d]
 ,
 :@{-}[u(1.5)] *{\circ}
 :@{-}"u"
 )
}
}
\end{xy}
$$
\bigskip
\bigskip
\bigskip
\bigskip

\end{proof}

\section{Algebras in and module categories over categories of representations of Hopf algebras}\label{halg}

Recall (e.g. from \cite{ka}) that the category $H\da Mod$ of (left) modules over a Hopf algebra $H$ is monoidal, with respect to the tensor product of (underlying) vector spaces. 
Indeed, the coproduct of $H$ gives rise to an $H$-module structure on the tensor product $X\otimes Y$ of two $H$-modules:
$$h(x\otimes y) = \Delta(h)(x\otimes y) = \sum_{(h)}h_{(1)}x\otimes h_{(2)}y,\quad h\in H, x\in X, y\in Y.$$
This monoidal structure allows us to define certain algebraic structures in the category $H\da Mod$. 
\newline
An {\em algebra} in the category $H\da M od$ (or an {\em $H$-module algebra} for short) is an associative algebra $A$ with a left $H$-module structure, such that 
$$h(ab) = \sum_{(h)}h_{(1)}(a)\otimes h_{(2)}(b),\quad h\in H, a,b\in A.$$
\newline
A (left) {\em module} over an $H$-module algebra $A$ in $H\da Mod$ is a left $A$-module $M$ together with a left $H$-module structure, such that 
\begin{equation}\label{hmod}
h(am) = \sum_{(h)}h_{(1)}(a)\otimes h_{(2)}(m),\quad h\in H, a\in A, m\in M.
\end{equation}
The category $_{A}(H\da Mod)$ of left modules over an $H$-module algebra $A$ in $H\da Mod$ is a right module category over $H\da Mod$, i.e. there is a categorical pairing
$$_{A}(H\da Mod)\times H\da Mod\to {_{A}}{(H\da Mod)},\quad (M,X)\mapsto M\otimes X,$$
with the diagonal (left) $H$-module structure on $M\otimes X$ and the (left) $A$-action:
$$A\otimes M\otimes X\to M\otimes X,\quad a\otimes m\otimes x\mapsto am\otimes x.$$

For an $H$-algebra $A$ define the {\em smash product} $A\#H$ as follows. As a vector space the smash product $A\#H$ is the tensor product of $A$ and $H$ (with decomposable tensor denoted by $a\#h$). The multiplication in $A\#H$ is given by
$$(a\#g)(b\#h) = \sum_{(g)}ag_{(1)}(b)\#g_{(2)}h,\quad g,h\in H, a,b\in A.$$
The following statement is well-known to the specialists (see Paragraph 1.7 of \cite{so}). 
\begin{lem}
The category $_{A}(H\da Mod)$ coincides with the category $A\#H\da Mod$ of (left) modules over the  smash product $A\#H$. 
\end{lem}
\begin{proof}
The $A\#H$-module structure on a left module $M$ over an $H$-module algebra $A$ in $H\da Mod$ is given by
$(a\#h)m = ah(m)$. Conversely, the embeddings $A,H\to A\#H$:
$$a\mapsto a\#1,\quad h\mapsto 1\#h$$ allow us to define $A$- and $H$-module structures on an $A\#H$-module, compatible in the way (\ref{hmod}). 
\end{proof}
Note that for any $H$-module algebra $A$ the smash product $B=A\#H$ is a (right) {\em $H$-comodule algebra}, i.e. that it is equipped with a homomorphism of algebras $\psi:B\to B\otimes H$, which is a (right) $H$-comodule structure. Indeed, the map $\psi:A\#H\to A\#H\otimes H$ can be defined as follows:
$$a\# h\mapsto \sum_{(h)}a\#h_{(1)}\otimes h_{(2)}.$$
Note also that the category of modules $B\da Mod$ over an $H$-comodule algebra is a right module category over $H\da Mod$:
$$B\da Mod\times H\da Mod\to B\da Mod,\quad (M,X)\mapsto M\otimes X,$$
where the $B$-module structure on $M\otimes X$ has the form:
$$b(m\otimes x) = \sum_{(b)}b_{(0)}m\otimes b_{(1)}x.$$ Here $\psi(b) = \sum_{(b)}b_{(0)}\otimes b_{(1)}\in B\otimes H$.

\begin{exam}Regular module category.
\end{exam}
The category $H\da Mod$ can always be considered as a right module category over itself (it also can be realised as the category of modules in $H\da Mod$ over the trivial $H$-module algebra $k$). The category $\End_{H\da Mod}(H\da Mod)$ of $H\da Mod$-invariant endofunctors of $H\da Mod$ coincides with $H\da Mod$. Indeed, any $H\da Mod$-invariant endofunctor of $H\da Mod$ is determined by its value on the monoidal unit $k$.

Similarly the category $H\da Comod$ of left $H$-comodules is a right module category over itself (again it also can be realised as the category of modules in $H\da Comod$ over the trivial $H$-comodule algebra $k$). The category $\End_{H\da Comod}(H\da Comod)$ of $H\da Comod$-invariant endofunctors of $H\da Comod$ coincides with $H\da Comod$.

\begin{exam}\label{ffmc} $Vect$ as a module category.
\end{exam}
The forgetful functor $H\da Mod\to\Vect$ turns $\Vect$ into a module category over $H\da Mod$. This category coincides with the category of modules in $H\da Mod$ over the $H$-module algebra $H^*$, with the (regular) $H$-action
$$h(l)(g) = l(gh),\quad l\in H^*, g,h\in H.$$
The category $\End_{H\da Mod}(\Vect)$ of $H\da Mod$-invariant endofunctors of $\Vect$ coincides with $H\da Comod$. Indeed, any $H\da Mod$-invariant endofunctor $F$ of $\Vect$ is given by its value $F(k)$, which has a natural structure of $H$-comodule (for details see e.g. \cite{os}, theorem 5).

Similarly the forgetful functor $H\da Comod\to\Vect$ turns $\Vect$ into a module category over $H\da Comod$. This time the category coincides with the category of modules in $H\da Comod$ over the $H$-comodule algebra $H$, with the (regular) $H$-coaction given by the comultiplication. 
\newline
The category $\End_{H\da Comod}(\Vect)$ of $H\da Comod$-invariant endofunctors of $\Vect$ coincides with $H\da Mod$. Indeed, any $H\da Comod$-invariant endofunctor $F$ of $\Vect$ is given by its value $F(k)$, which has a natural structure of $H$-module.
\bigskip

The two examples above are extreme cases of the following situation. Let $H\to F$ be an epimorhism of Hopf subalgebra. The induction functor $ind^H_F\se H\da Comod\to F\da Comod$ is monoidal and in particular makes $F\da Comod$ a module category over $H\da Comod$. Here we describe this module category as the category of modules in $H\da Comod$ over an $H$-comodule algebra. The algebra has the following form 
$$A(H,F) = \{h\in H|\ \Delta(h) \in H\otimes F \},$$
with the $H$-coaction given by $\Delta$. 
\newline
The functor $\F\da Comod\to{_A}H\da Comod$ sends an $F$-comodule $N$ with $F$-comodule structure $\psi\se N\to F\otimes N$ into 
$$(H\otimes N)^F = \{x\in H\otimes L|\ (1\otimes t)( \Delta\otimes 1)(x) = (1\otimes\psi)(x),\quad \forall f\in F\}.$$ The $H$-comodule structure on $(H\otimes N)^F$ has the form $h*x = x(S(h)\otimes 1)$.

\section{Monoidal centre. Yetter-Drinfeld modules and Drinfeld double}\label{moncen}

Here, after recalling (e.g. from \cite{ka}) the descriptions of the monoidal centre $\cZ(H\da Mod)$ in terms of Yetter-Drinfeld modules and Drinfeld double, we look at commutative algebras in $\cZ(H\da Mod)$.

A {\em Yetter-Drinfeld} module over a Hopf algebra $H$ is a left $H$-module $M$ with a left $H$-comodule structure, satisfying the following compatibility condition:
\begin{equation}\label{com}
 \sum_{(h),(m)}h_{(1)}m_{(-1)}\otimes h_{(2)}m_{(0)} = \sum_{(h),(m)}(h_{(1)}m)_{(-1)}h_{(2)}\otimes (h_{(1)}m)_{(0)},
 \end{equation}
 which in diagrammatic notation has the form:
 \bigskip
 $$
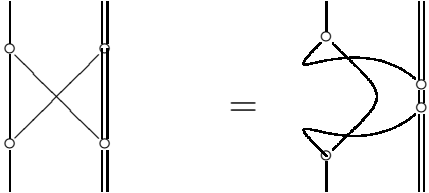
\begin{xy} (-20,0)*+!R{\xygraph{ !{0;/r1.5pc/:;/u1.5pc/::}
*{} :@{-}[d] *{\circ}
(
:@{-}[dd] *{\circ}
 (
 :@{-}[d]
 ,
 :@{-}[uurr]="r" *{\circ}
 :@{=}[u]
 )
,
:@{-}[ddrr] *{\circ}
 (
 :@{=}"r"
 ,
 :@{=}[d]
 )
}
}
\POS(10,-8)*{{\eq}}
\POS(20,0)*+!L{\xygraph{ !{0;/r1.5pc/:;/u1.5pc/::}
*{} :@{-}[d(.75)] *{\circ}
(
:@{-}@(rd,ru)[d(2.5)]="d" *{\circ}
:@{-}[d(.75)]
,
:@{-}@(dl,ul)[drr] *{\circ}
(
:@{=}[u(1.75)]
,
:@{=}[d(.5)] *{\circ}
(
:@{=}[d(1.75)]
,
:@{-}@(dl,ul)"d"
)
)
}
}
\end{xy}
$$
\bigskip\bigskip
\bigskip
\newline
Later on we will need the following equivalent form of the Yetter-Drinfeld condition (\ref{com}):
\begin{equation}\label{yde}
\sum_{(hm)}(hm)_{(-1)}\otimes(hm)_{(0)} = \sum_{(h),(m)}h_{(1)}m_{(-1)}S(h_{(3)})\otimes h_{(2)}m_{(0)}
\end{equation}
which has the following string presentation:

\medskip
$$
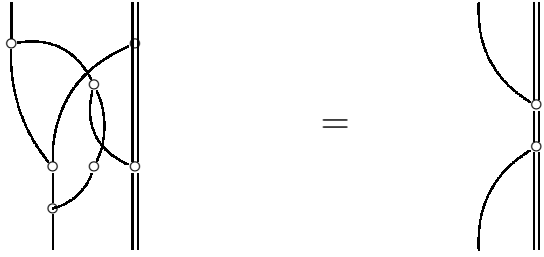
\begin{xy} (-65,0)*+!R{
\xygraph{ !{0;/r1.3pc/:;/u1.3pc/::}
*{} :@{-}[d] *{\circ}
(
:@{-}@/_5pt/[dddr] *{\circ}
 (
 :@{-}@/^10pt/[uuurr]="u" *{\circ}
 ,
 :@{-}[d]="d" *{\circ}
 :@{-}[d] 
 )
,
:@{-}@/^7pt/[drr] *{\circ}
 (
 :@{-}@/^4pt/[dd] *{\circ}
 :@{-}@/^3pt/"d"
 ,
 :@{-}@/_8pt/[ddr] *{\circ}
  (
  :@{=}"u"
  :@{=}[u]
  ,
  :@{=}[dd]
  )
 )
)
}
}
\POS(-23,-11)*{\eq}
\POS(-5,0)*+!L{
\xygraph{ !{0;/r1.3pc/:;/u1.3pc/::}
*{} :@{-}@/_7pt/[d(2.5)r(1.4)] *{\circ}
(
:@{=}[d(1)] *{\circ}
 (
 :@{-}@/_7pt/[d(2.5)l(1.4)]
 ,
 :@{=}[d(2.5)]
 )
,
:@{=}[u(2.5)] 
)
}
}
\end{xy}
$$
\bigskip
\bigskip
\bigskip
\bigskip
\newline
and can be proved by the following sequence of moves:

$$
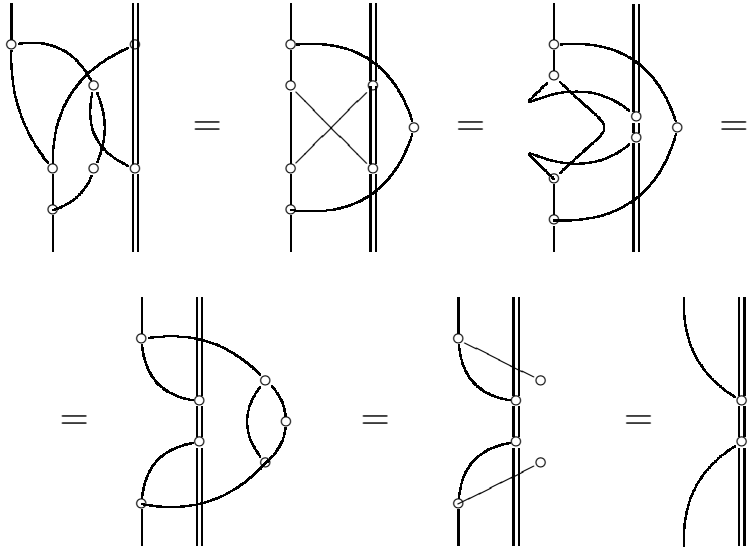
\begin{xy} (-135,0)*+!R{
\xygraph{ !{0;/r1.3pc/:;/u1.3pc/::}
*{} :@{-}[d] *{\circ}
(
:@{-}@/_5pt/[dddr] *{\circ}
 (
 :@{-}@/^10pt/[uuurr]="u" *{\circ}
 ,
 :@{-}[d]="d" *{\circ}
 :@{-}[d] 
 )
,
:@{-}@/^7pt/[drr] *{\circ}
 (
 :@{-}@/^4pt/[dd] *{\circ}
 :@{-}@/^3pt/"d"
 ,
 :@{-}@/_8pt/[ddr] *{\circ}
  (
  :@{=}"u"
  :@{=}[u]
  ,
  :@{=}[dd]
  )
 )
)
}
}
\POS(-110,-11)*{\eq}
\POS(-100,0)*+!L{
\xygraph{ !{0;/r1.3pc/:;/u1.3pc/::}
*{} :@{-}[d] *{\circ}
(
:@{-}[d] *{\circ}
(
:@{-}[dd] *{\circ}
 (
 :@{-}[d]="d" *{\circ}
 :@{-}[d]
 ,
 :@{-}[uurr]="r" *{\circ}
 :@{=}[uu]
 )
,
:@{-}[ddrr] *{\circ}
 (
 :@{=}"r"
 ,
 :@{=}[dd]
 )
 )
,
:@{-}@/^12pt/[ddrrr] *{\circ}
:@{-}@/^12pt/"d"
)
}
}
\POS(-75,-11)*{\eq}
\POS(-65,0)*+!L{
\xygraph{ !{0;/r1.3pc/:;/u1.3pc/::}
*{} :@{-}[d] *{\circ}
(
:@{-}[d(.75)] *{\circ}
(
:@{-}@(rd,ru)[d(2.5)]="d" *{\circ}
:@{-}[d]="dd" *{\circ}
:@{-}[d(.75)]
,
:@{-}@(dl,ul)[drr] *{\circ}
(
:@{=}[u(2.75)]
,
:@{=}[d(.5)] *{\circ}
(
:@{=}[d(2.75)]
,
:@{-}@(dl,ul)"d"
)
)
)
,
:@{-}@/^12pt/[ddrrr] *{\circ}
:@{-}@/^12pt/"dd"
)
)
}
}
\POS(-40,-11)*{\eq}
\end{xy}
$$
\bigskip
\bigskip
\bigskip
\bigskip

$$\begin{xy} 
\POS(-145,-11)*{\eq}
\POS(-135,0)*+!R{
\xygraph{ !{0;/r1.3pc/:;/u1.3pc/::}
*{} :@{-}[d] *{\circ}
(
:@{-}@/_7pt/[d(1.5)r(1.4)] *{\circ}
(
:@{=}[d(1)] *{\circ}
 (
 :@{-}@/_7pt/[d(1.5)l(1.4)]="d" *{\circ}
 :@{-}[d]
 ,
 :@{=}[d(2.5)]
 )
,
:@{=}[u(2.5)] 
)
,
:@{-}@/^7pt/[drrr] *{\circ}
(
:@{-}@/_7pt/[dd]="r" *{\circ}
,
:@{-}@/^2pt/[d(1)r(.5)] *{\circ}
:@{-}@/^2pt/"r"
:@{-}@/^7pt/"d"
)
)
}
}
\POS(-105,-11)*{\eq}
\POS(-95,0)*+!L{
\xygraph{ !{0;/r1.3pc/:;/u1.3pc/::}
*{} :@{-}[d] *{\circ}
(
:@{-}@/_7pt/[d(1.5)r(1.4)] *{\circ}
(
:@{=}[d(1)] *{\circ}
 (
 :@{-}@/_7pt/[d(1.5)l(1.4)]="d" *{\circ}
 :@{-}[d]
 ,
 :@{=}[d(2.5)]
 )
,
:@{=}[u(2.5)] 
)
,
:@{-}[drr] *{\circ}
:@{}[dd] *{\circ}
:@{-}"d"
)
}
}
\POS(-70,-11)*{\eq}
\POS(-65,0)*+!L{
\xygraph{ !{0;/r1.3pc/:;/u1.3pc/::}
 !{0;/r1.3pc/:;/u1.3pc/::}
*{} :@{-}@/_7pt/[d(2.5)r(1.4)] *{\circ}
(
:@{=}[d(1)] *{\circ}
 (
 :@{-}@/_7pt/[d(2.5)l(1.4)]
 ,
 :@{=}[d(2.5)]
 )
,
:@{=}[u(2.5)] 
)
}
}
\end{xy}
$$
\bigskip
\bigskip
\bigskip
\bigskip
 
The condition (\ref{com}) is preserved by the tensor product (of modules and comodules):

\medskip
$$
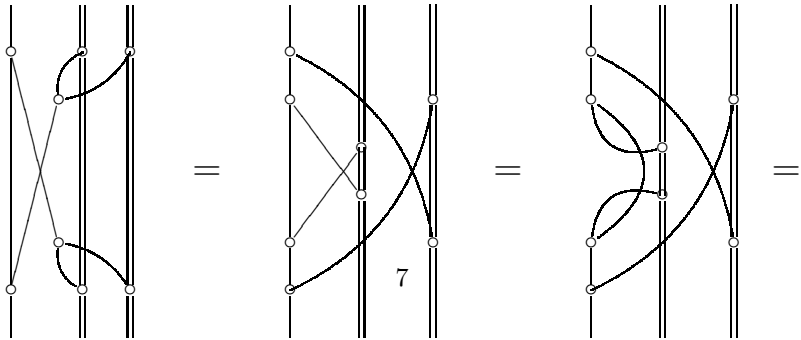
\begin{xy} (-135,0)*+!R{
\xygraph{ !{0;/r1.5pc/:;/u1.5pc/::}
*{} :@{-}[d] *{\circ}
(
:@{-}[d(4)r] *{\circ}
 (
 :@{-}@/_4pt/[d(1)r(.5)] *{\circ}
  (
  :@{=}[d]
  ,
  :@{=}[u(5)]="u" *{\circ}
  :@{=}[u]
  )
 ,
 :@{-}@/^4pt/[d(1)r(1.5)] *{\circ}
  (
  :@{=}[d]
  ,
  :@{=}[u(5)]="r" *{\circ}
  :@{=}[u]
  )
 )
,
:@{-}[d(5)] *{\circ}
 (
 :@{-}[d]
 ,
 :@{-}[u(4)r] *{\circ}
  (
  :@{-}@/^4pt/"u"
  ,
  :@{-}@/_4pt/"r"
  )
 )
)
}
}
\POS(-110,-16)*{\eq}
\POS(-100,0)*+!L{
\xygraph{ !{0;/r1.5pc/:;/u1.5pc/::}
*{} :@{-}[d] *{\circ}
(
: @{-}[d] *{\circ}
 (
 :@{-}[d(3)] *{\circ}
  (
  :@{-}[u(2)r(1.5)]="m" *{\circ}
  :@{=}[u(3)]
  ,
  :@{-}[d]="l" *{\circ}
  :@{-}[d]
  )
 ,
 :@{-}[d(2)r(1.5)] *{\circ}
  (
  :@{=}"m"
  ,
  :@{=}[d(3)]
  )
 )
,
:@{-}@/^12pt/[d(4)r(3)] *{\circ}
 (
 :@{=}[d(2)]
 ,
 :@{=}[u(3)] *{\circ}
  (
  :@{=}[u(2)]
  ,
  :@{-}@/^12pt/"l"
  )
 )
)
}
}
\POS(-70,-16)*{\eq}
\POS(-60,0)*+!L{
\xygraph{ !{0;/r1.5pc/:;/u1.5pc/::}
*{} :@{-}[d] *{\circ}
(
: @{-}[d] *{\circ}
 (
 :@{-}@/^20pt/[d(3)] *{\circ}
  (
  :@{-}@/^10pt/[u(1)r(1.5)]="m" *{\circ}
  :@{=}[d(3)]
  ,
  :@{-}[d]="l" *{\circ}
  :@{-}[d]
  )
 ,
 :@{-}@/_10pt/[d(1)r(1.5)] *{\circ}
  (
  :@{=}"m"
  ,
  :@{=}[u(3)]
  )
 )
,
:@{-}@/^12pt/[d(4)r(3)] *{\circ}
 (
 :@{=}[d(2)]
 ,
 :@{=}[u(3)] *{\circ}
  (
  :@{=}[u(2)]
  ,
  :@{-}@/^12pt/"l"
  )
 )
)
}
}
\POS(-33,-16)*{\eq}
\end{xy}
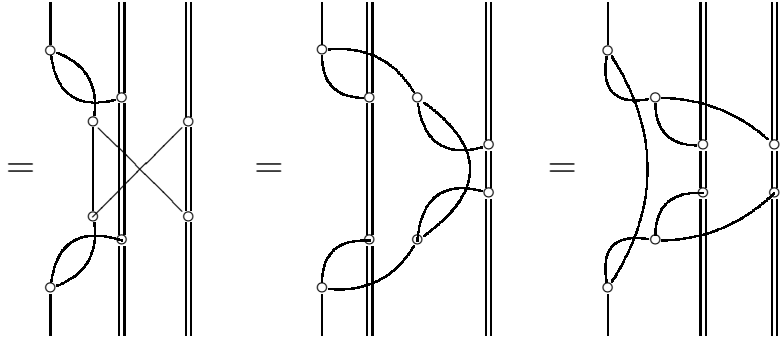
$$
\bigskip
\bigskip
\bigskip
\bigskip
\bigskip
\bigskip

$$\begin{xy} 
\POS(-167,-16)*{\eq}
\POS(-162,0)*+!R{
\xygraph{ !{0;/r1.5pc/:;/u1.5pc/::}
*{} :@{-}[d] *{\circ}
(
:@{-}@/_10pt/[r(1.5)d] *{\circ}
 (
 :@{=}[uu]
 ,
 :@{=}[ddd]="d" *{\circ} 
 :@{=}[dd]
 )
,
:@{-}@/^7pt/[r(.9)d(1.5)] *{\circ}
  (
  :@{-}[dd]="m" *{\circ}
  :@{-}@/^7pt/[l(.9)d(1.5)]="l" *{\circ}
   (
   :@{-}[d]
   ,
   :@{-}@/^10pt/"d"
   )
  ,
  :@{-}[ddrr] *{\circ}
   (
   :@{=}[d(2.5)]
   ,
   :@{=}[uu] *{\circ}
    (
    :@{=}[u(2.5)]
    ,
    :@{-}"m"
    )
   )
  )
 )
}
}
\POS(-134,-16)*{\eq}
\POS(-128,0)*+!L{
\xygraph{ !{0;/r1.5pc/:;/u1.5pc/::}
*{} :@{-}[d] *{\circ}
(
:@{-}@/_7pt/[dr] *{\circ}
 (
 :@{=}[uu]
 ,
 :@{=}[d(3)]="l" *{\circ}
 :@{=}[dd]
 )
,
:@{-}@/^7pt/[drr] *{\circ}
 (
 :@{-}@/_10pt/[d(1)r(1.5)] *{\circ}
  (
  :@{=}[u(3)]
  ,
  :@{=}[d] *{\circ}
   (
   :@{=}[d(3)]
   ,
   :@{-}@/_10pt/[l(1.5)d]
   )
  )
 ,
 :@{-}@/^20pt/[d(3)]="m" *{\circ}
 :@{-}@/^7pt/[dll] *{\circ}
  (
  :@{-}@/^7pt/"l"
  ,
  :@{-}[d]
  )
 )
)
}
}
\POS(-95,-16)*{\eq}
\POS(-90,0)*+!L{
\xygraph{ !{0;/r1.5pc/:;/u1.5pc/::}
*{} :@{-}[d] *{\circ}
(
:@{-}@/_10pt/[d(1)r(1)] *{\circ}
 (
 :@{-}@/_7pt/[dr] *{\circ}
  (
  :@{=}[u(3)]
  ,
  :@{=}[d]="l" *{\circ}
  :@{=}[d(3)]
  )
 ,
 :@{-}@/^5pt/[d(1)r(2.5)] *{\circ}
  (
  :@{=}[u(3)]
  ,
  :@{=}[d]="r" *{\circ}
  :@{=}[d(3)]
  )
 )
,
:@{-}@/^15pt/[d(5)] *{\circ}
 (
 :@{-}@/^10pt/[u(1)r(1)] *{\circ}
  (
  :@{-}@/^7pt/"l"
  ,
  :@{-}@/_5pt/"r"
  )
 ,
 :@{-}[d]
 )
)
}
}
\end{xy}
$$
\bigskip
\bigskip
\bigskip
\bigskip\bigskip\bigskip
\newline
thus making $\YD(H)$ a monoidal category. Moreover this category is braided, with the braiding
$$c_{M.N}(m\otimes n) = \sum_{(m)}m_{(-1)}n\otimes m_{(0)},$$
which graphically is represented by the picture: 
$\quad\quad
\xygraph{ !{0;/r1.0pc/:;/u1.0pc/::}
*{} :@{=}[dr] *{\circ}
(
:@{-}[dd] *{\circ}
 (
 :@{=}[ld]
 ,
 :@{=}[uurr]
 :@{=}[ur]
 )
,
:@{=}[dddrrr] 
}
$
\medskip
\newline
It is invertible with the inverse given by

The hexagon (triangle) axioms take the shape:

$$
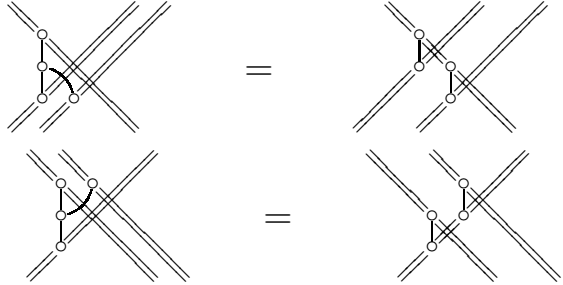
\begin{xy} (-55,0)*+!R{
\xygraph{ !{0;/r1.0pc/:;/u1.0pc/::}
*{} :@{=}[dr] *{\circ}
(
:@{=}[dddrrr] 
,
:@{-}[d] *{\circ}
(
:@{-}[d] *{\circ}
(
 :@{=}[ld]
 ,
 :@{=}[uuurrr]
 )
,
:@{-}@/^3pt/[dr] *{\circ}
(
:@{=}[dl]
,
:@{=}[uuurrr]
)
)
}
}
\POS(-23,-5)*{{\eq}}
\POS(-7,0)*+!L{
\xygraph{ !{0;/r1.0pc/:;/u1.0pc/::}
*{} :@{=}[dr] *{\circ}
(
:@{=}[dr] *{\circ}
(
:@{-}[d] *{\circ}
(
:@{=}[dl]
,
:@{=}[uuurrr]
)
,
:@{=}[ddrr]
)
,
:@{-}[d] *{\circ}
(
:@{=}[ddll] 
,
:@{=}[uurr]
)
)
}
}
\end{xy}
$$
\bigskip

$$\begin{xy} (-55,0)*+!R{
\xygraph{ !{0;/r1.0pc/:;/u1.0pc/::}
*{} :@{=}[dr] *{\circ}
(
:@{=}[dddrrr] 
,
:@{-}[d] *{\circ}
(
:@{-}[d] *{\circ}
 (
 :@{=}[uuurrr]
 ,
 :@{=}[dl]
 )
 ,
:@{-}@/_3pt/[ur] *{\circ}
 (
 :@{=}[ul]
 ,
 :@{=}[dddrrr]
 )
)
)
}
}
\POS(-23,-5)*{{\eq}}
\POS(-12,0)*+!L{
\xygraph{ !{0;/r1.0pc/:;/u1.0pc/::}
*{} :@{=}[ddrr] *{\circ}
(
:@{-}[d] *{\circ}
(
:@{=}[dl]
,
:@{=}[ur] *{\circ}
(
:@{=}[uurr]
,
:@{-}[u] *{\circ}
(
:@{=}[ul]
,
:@{=}[dddrrr]
)
)
)
,
:@{=}[ddrr]
)
}
}
\end{xy}
$$
\bigskip

The importance of Yetter-Drinfeld modules come from the fact that $\YD(H)$ is equivalent to the monoidal centres $\cZ(H\da Mod)$, $\cZ(H\da Comod)$. 
\newline
The functor $\cZ(H\da Mod)\to \YD(H)$ sends $(Z,z)\in\cZ(H\da Mod)$ into the $H$-module with the $H$-coaction $\psi(z) = z_H(z\otimes 1)$. Here $z_X:Z\otimes X\to X\otimes Z$ is the half braiding of $(Z,z)\in\cZ(H\da Mod)$ and $z_H$ is its specialisation on $H\in H\da Mod$. 
\newline
Conversely, the functor $\YD(H)\to\cZ(H\da Mod)$ defines a half braiding $M\otimes X\to X\otimes M$ 
$$m\otimes x\mapsto \sum_{(m)}m_{(-1)}x\otimes m_{(0)}$$ on a Yetter-Drinfeld module $M$. $H$-linearity of the half braiding is equivalent to the condition (\ref{com}), see \cite{ka} fo details.
\newline
Similarly, the functor $\cZ(H\da Comod)\to \YD(H)$ sends $(Z,z)\in\cZ(H\da Mod)$ into the $H$-comodule with the $H$-action $hz = (\varepsilon\otimes 1)z_H(z\otimes h)$. Here $z_X:Z\otimes X\to X\otimes Z$ is the half braiding of $(Z,z)\in\cZ(H\da Comod)$ and $z_H$ is its specialisation on $H\in H\da Comod$ and $\varepsilon:H\to k$ is the counit. 
\newline
Conversely, the functor $\YD(H)\to\cZ(H\da Comod)$ defines a half braiding $M\otimes L\to L\otimes M$ 
$$m\otimes l\mapsto \sum_{(l)}m_{(0)}\otimes S(m_{(-1)})m$$ on a Yetter-Drinfeld module $M$. 

For a finite dimensional $H$ the category $\YD(H)$ can be realised as the category $D(H)\da Mod$ of modules over a Hopf algebra $D(H)$, the {\em Drinfeld double} of $H$. As a vector space $D(H)$ is the tensor product $H\otimes H^*$ (with elements denoted by $hl,\ l\in H^*, h\in H$). Moreover $H$ and $H^*$ are Hopf subalgebras of $D(H)$. The multiplication and comultiplication are given by
$$lh = \sum_{(h)}h_{(2)}l(h_{(1)}\da S(h_{(3)})),\quad \Delta(hl) = \sum_{(l),(h)}h_{(1)}l_{(1)}\otimes h_{(2)}l_{(2)}.$$
Indeed the $H$-module part of a Yetter-Drinfeld module $M$ correspond to the action of $H\subset D(H)$, while the $H$-coaction comes from the action of $H^*\subset D(H)$. The consequence (\ref{yde}) 
of the compatibility condition (\ref{com}) implies the above formula for the product in $D(H)$ (again see \cite{ka} for details). 

To make a distinction with the commutativity in $\Vect$ we call a commutative algebra $A$ in $\YD(H)$ quantum commutative. Explicitly, an algebra $A\in\YD(H)$ is  {\em quantum commutative} if 
$$ab =  \sum_{(a)}a_{(-1)}(b)a_{(0)},\quad a,b\in A,$$ where $a\mapsto \sum_{(a)}a_{(-1)}\otimes a_{(0)}\in H\otimes A$ is the comodule structure on $A$. 
\newline
Following \cite{da,os} define for a not necessarily commutative algebra $A\in\YD$ its {\em left centre} by
$$C_l(A) = \{a\in A|\ ab=\sum_{(a)}a_{(-1)}(b)a_{(0)},\quad \forall b\in A\}.$$
The fact that this is a submodule and subcomodule (and hence a subobject in $\YD(H)$) of $A$ follows from $H$-(co)liearity of the braiding in $\YD(H)$. The fact that this is a quantum commutative subalgebra of $A$ follows from its definition (see also \cite{da} for a general categorical explanation). 

\section{Full centre}\label{}

The forgetful functor $F\se \cZ(H\da Mod)\to H\da Mod$ corresponds to the forgetful functor $\YD(H)\to H\da Mod$. Here we describe its right adjoint, which is an important technical statement needed for the proof of the main result of the paper. Note that this right adjoint was described in \cite{cm} (corollary 2.8)  (see also Proposition 1 of \cite{ra} and Paragraph 7 of \cite{sz}). Here we add a proof for the sake of completeness. 
\begin{prop}\label{rad}
The forgetful functor $\YD(H)\to H\da Mod$ has a right adjoint $R(\da)=H\otimes\da\se H\da Mod\to \YD(H)$, where $H$-action and $H$-coaction on $R(N)=H\otimes N$ have the form:
$$h(g\otimes n) = \sum_{(h)}h_{(1)}gS(h_{(3)})\otimes h_{(2)}n,\quad \psi(g\otimes n) = \sum_{(g)}g_{(1)}\otimes g_{(2)}\otimes n\in H\otimes H\otimes N.$$
The adjunction natural transformations 
$$\alpha_M\se M\to RF(M),\quad\beta_L\se FR(L)\to L$$ are given by the comodule structure on an object of $M\in\YD(H)$ and by the counit respectively:
$$\alpha_M(m) = \sum_{(m)}m_{(-1)}\otimes m_{(0)},\quad \beta_N(h\otimes n) = \varepsilon(h)n.$$
In particular, the adjunction morphism $\beta$ is epi. 
\end{prop}
\begin{proof}
The string diagrams of $H$-action and $H$-coaction on $R(N)=H\otimes N$ have the form:
\medskip
$$
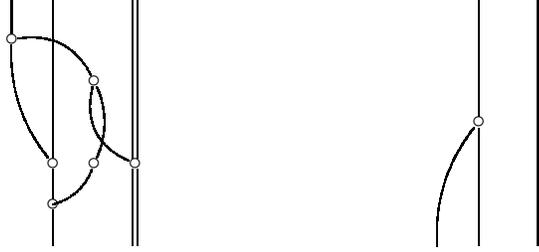
\begin{xy} (-65,0)*+!R{
\xygraph{ !{0;/r1.3pc/:;/u1.3pc/::}
*{} :@{-}[d] *{\circ}
(
:@{-}@/_5pt/[dddr] *{\circ}
 (
 :@{-}[uuuu]
 ,
 :@{-}[d]="d" *{\circ}
 :@{-}[d]
 )
,
:@{-}@/^7pt/[drr] *{\circ}
 (
 :@{-}@/^4pt/[dd] *{\circ}
 :@{-}@/^3pt/"d"
 ,
 :@{-}@/_8pt/[ddr] *{\circ}
  (
  :@{=}[uuuu]
  ,
  :@{=}[dd]
  )
 )
)
}
}
\POS(-23,-5)*{}
\POS(-5,0)*+!L{
\xygraph{ !{0;/r1.3pc/:;/u1.3pc/::}
*{} :@{-}[d(3)] *{\circ}
(
:@{-}@/_5pt/[d(3)l] 
,
:@{-}[d(3)] 
:@{}[r(1.5)]
:@{=}[u(6)]
)
}
}
\end{xy}
$$
\bigskip\bigskip\bigskip\bigskip\bigskip\bigskip
\newline
The Yetter-Drinfeld condition (\ref{com}) for these action and coaction can be proved as follows:

$$
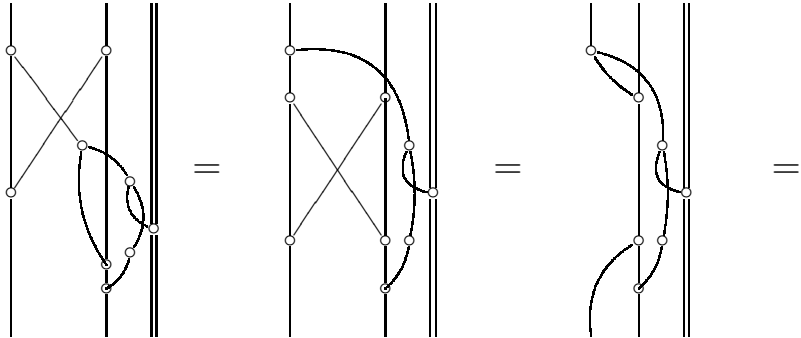
\begin{xy} (-135,0)*+!R{
\xygraph{ !{0;/r1.5pc/:;/u1.5pc/::}
*{} :@{-}[d] *{\circ}
(
:@{-}[d(3)] *{\circ}
 (
 :@{-}[d(3)]
 ,
 :@{-}[u(3)r(2)] *{\circ}
  (
  :@{-}[u]
  ,
  :@{-}[d(4.5)]="m" *{\circ}
  :@{-}[d(.5)]="d" *{\circ}
  :@{-}[d]
  )
 )
,
:@{-}[d(2)r(1.5)] *{\circ}
 (
 :@{-}@/_5pt/"m"
 ,
 :@{-}@/^3pt/[d(.75)r(1)] *{\circ}
  (
  :@{-}@/^5pt/[d(1.5)] *{\circ}
  :@{-}@/^2pt/"d"
  ,
  :@{-}@/_5pt/[d(1)r(.5)] *{\circ}
   (
   :@{=}[d(2.25)]
   ,
   :@{=}[u(4.75)]
   )
  )
 )
)
}
}
\POS(-110,-16)*{\eq}
\POS(-100,0)*+!L{
\xygraph{ !{0;/r1.5pc/:;/u1.5pc/::}
*{} :@{-}[d] *{\circ}
(
:@{-}[d] *{\circ}
 (
 :@{-}[d(3)] *{\circ}
  (
  :@{-}[d(2)]
  ,
  :@{-}[u(3)r(2)]="m" *{\circ}
  :@{-}[u(2)]
  )
 ,
 :@{-}[d(3)r(2)] *{\circ}
  (
  :@{-}"m"
  ,
  :@{-}[d]="d" *{\circ}
  :@{-}[d]
  )
 )
,
:@{-}@/^15pt/[d(2)r(2.5)] *{\circ}
 (
 :@{-}@/^2pt/[d(2)] *{\circ}
 :@{-}@/^2pt/"d"
 ,
 :@{-}@/_7pt/[d(1)r(.5)] *{\circ}
  (
  :@{=}[u(4)]
  ,
  :@{=}[d(3)]
  )
 )
)
}
}
\POS(-70,-16)*{\eq}
\POS(-60,0)*+!L{
\xygraph{ !{0;/r1.5pc/:;/u1.5pc/::}
*{} :@{-}[d] *{\circ}
(
:@{-}@/_2pt/[dr] *{\circ}
 (
 :@{-}[u(2)]
 ,
 :@{-}[d(3)] *{\circ}
  (
  :@{-}@/_7pt/[d(2)l]
  ,
  :@{-}[d]="d" *{\circ}
  :@{-}[d]
  )
 )
,
:@{-}@/^10pt/[d(2)r(1.5)] *{\circ}
 (
 :@{-}@/^2pt/[dd] *{\circ}
 :@{-}@/^2pt/"d"
 ,
 :@{-}@/_7pt/[d(1)r(.5)] *{\circ}
  (
  :@{=}[u(4)]
  ,
  :@{=}[d(3)]
  )
 )
}
}
\POS(-33,-16)*{\eq}
\end{xy}
$$
\bigskip
\bigskip
\bigskip
\bigskip
\bigskip
\bigskip
\newline
then by lemma \ref{aux} 

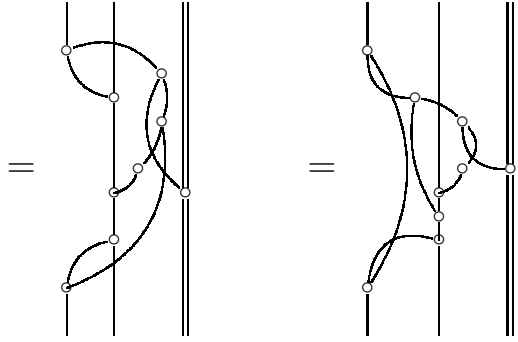
$$\begin{xy} 
\POS(-115,-16)*{\eq}
\POS(-110,0)*+!L{
\xygraph{ !{0;/r1.5pc/:;/u1.5pc/::}
*{} :@{-}[d] *{\circ}
(
:@{-}@/_5pt/[dr] *{\circ}
 (
 :@{-}[uu]
 ,
 :@{-}[dd]="m" *{\circ}
 :@{-}[d] *{\circ}
  (
  :@{-}@/_5pt/[dl]="d" *{\circ}
  :@{-}[d]
  ,
  :@{-}[dd]
  )
 )
,
:@{-}@/^7pt/[r(2)d(.5)] *{\circ}
 (
 :@{-}@/^2pt/[d] *{\circ}
  (
  :@{-}@/^2pt/[d(1)l(.5)] *{\circ}
  :@{-}@/^2pt/"m"
  ,
  :@{-}@/^15pt/"d"
  )
 ,
 :@{-}@/_10pt/[r(.5)d(2.5)] *{\circ}
  (
  :@{=}[u(4)]
  ,
  :@{=}[d(3)]
  )
 )
)
)
}
}
\POS(-75,-16)*{\eq}
\POS(-70,0)*+!L{
\xygraph{ !{0;/r1.5pc/:;/u1.5pc/::}
*{} :@{-}[d] *{\circ}
(
:@{-}@/^15pt/[d(5)] *{\circ}
 (
 :@{-}[d]
 ,
 :@{-}@/^10pt/[u(1)r(1.5)]="m" *{\circ}
 :@{-}[dd]
 )
,
:@{-}@/_7pt/[rd] *{\circ}
 (
 :@{-}@/_5pt/[d(2.5)r(.5)] *{\circ}
  (
  :@{-}[u(.5)]="mu" *{\circ} 
  :@{-}[u(4)]
  ,
  :@{-}"m"
  )
 ,
 :@{-}@/^2pt/[r(1)d(.5)] *{\circ}
  (
  :@{-}@/^5pt/[d] *{\circ}
  :@{-}@/^2pt/"mu"
  ,
  :@{-}@/_7pt/[rd] *{\circ}
   (
   :@{=}[u(3.5)]
   ,
   :@{=}[d(3.5)]
   )
  )
 )
)
}
}
\end{xy}
$$
\bigskip
\bigskip
\bigskip
\bigskip
\bigskip
\bigskip

The $H$-linearity of $\beta$ is transparent:
\medskip
$$
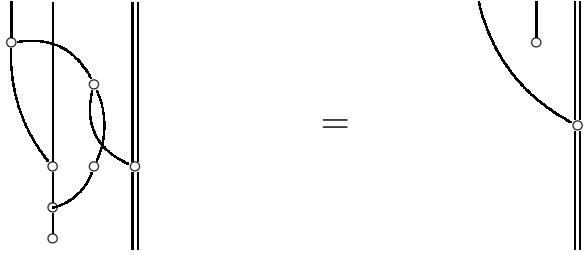
\begin{xy} (-65,0)*+!R{
\xygraph{ !{0;/r1.3pc/:;/u1.3pc/::}
*{} :@{-}[d] *{\circ}
(
:@{-}@/_5pt/[dddr] *{\circ}
 (
 :@{-}[uuuu]
 ,
 :@{-}[d]="d" *{\circ}
 :@{-}[d(.75)] *{\circ}
 )
,
:@{-}@/^7pt/[drr] *{\circ}
 (
 :@{-}@/^4pt/[dd] *{\circ}
 :@{-}@/^3pt/"d"
 ,
 :@{-}@/_8pt/[ddr] *{\circ}
  (
  :@{=}[uuuu]
  ,
  :@{=}[dd]
  )
 )
)
}
}
\POS(-23,-11)*{\eq}
\POS(-5,0)*+!L{
\xygraph{ !{0;/r1.3pc/:;/u1.3pc/::}
*{} :@{-}@/_7pt/[d(3)r(2,4)] *{\circ}
(
:@{=}[d(3)] 
,
:@{=}[u(3)] 
:@{}[l]
:@{-}[d] *{\circ}
)
}
}
\end{xy}
$$
\bigskip\bigskip\bigskip\bigskip
\newline
The $H$-colinearity of $\alpha_M$ is equivalent to the coassociativity of the $H$-comodule structure of $M$. The $H$-linearity of $\alpha$ follows from the identity (\ref{yde}). 

The adjunction axioms for $\alpha$ and $\beta$ follow from the counit axiom. 
\end{proof}
\begin{rem}
\end{rem}
For a finite dimensional $H$ the forgetful functor $\cZ(H\da Mod)\to H\da Mod$ corresponds to the restriction functor $D(H)\da Mod\to H\da Mod$ along the embedding $H\to D(H)$. In particular it has a right adjoint $$R(\da)=Hom_H(D(H),\da)\se H\da Mod\to D(H)\da Mod,$$
which due to the (multiplicative) decomposition $D(H)=H.H^*$, coincides with the adjoint functor from proposition \ref{rad}:
$$Hom_H(D(H),\da) = Hom_H(H.H^*,\da) = Hom(H^*,\da) = H\otimes\da.$$
\medskip

As a right adjoint to a monoidal functor the functor $R\se H\da Mod\to \YD(H)$ has a lax (op)monoidal structure $R(M)\otimes R(N)\to R(M\otimes N)$ (e.g. see \cite{da}), which in our case has the form
$$H\otimes M\otimes H\otimes N\to H\otimes M\otimes N,\quad g\otimes m\otimes h\otimes n\mapsto gh\otimes m\otimes n.$$
In particular the functor $R$ sends an $H$-module algebra $A$ into an algebra $R(A)=H\otimes A$ in the monoidal category $\YD(H)$, where the multiplcation in $H\otimes A$ is just the tensor product multiplication:
$$(g\otimes a)(h\otimes b) = gh\otimes ab,\quad g,h\in H, a,b\in A.$$
In \cite{da} we derived a formula (theorem 5.4) expressing the full centre in terms of the right adjoint $R$ to the forgetful functor $F:\cZ(\cC)\to \cC$. The condition for this formula to work was epimorphity of the adjunction morphism 
$$\beta_X:FR(X)\to X,\quad U\in\cZ(\cC), X\in\cC.$$
In our case the adjunction map is epi (by proposition \ref{rad}), which allows us to get the main result of this note.
\begin{theo}\label{main}
The full centre $Z(A)$ of an $H$-module algebra $A$ coincides with the centraliser $C_{A\#H}(A)$. The Yetter-Drinfeld structure ($H$-action and $H$-coaction) on $C_{A\#H}(A)$ has the form:
$$h(\sum_ia_i\# g_i) = \sum_{(h),i}h_{(2)}(a_i)\#S^2(h_{(3)})g_iS(h_{(1)}),$$ $$\psi(\sum_ia_i\# g_i) = \sum_{i,(g_i)}S^{-1}((g_i)_{(2)})\otimes a_i\#(g_i)_{(1)},$$ where $\sum_ia_i\# g_i\in C_{A\#H}(A)$.
\end{theo}
\begin{proof}
By \cite{da} we can identify $Z(A)$ with the left centre $C_l(R(A))$. By the definition of left centre (see section \ref{moncen}), $\sum_ig_i\otimes a_i$ belongs to $C_l(R(A))$ if for all $h\in H$, $b\in A$
$$\sum_{i}g_ih\otimes a_ib = \sum_{i,(g_i)}(g_i)_{(1)}(h\otimes b).((g_i)_{(2)}\otimes a_i) =$$ 
$$\sum_{i,(g_i)}((g_i)_{(1)}hS((g_i)_{(3)})\otimes (g_i)_{(2)}(b)).((g_i)_{(4)}\otimes a_i) = \sum_{i,(g_i)}(g_i)_{(1)}h\otimes (g_i)_{(2)}(b)a_i$$
or simply if 
$$\sum_{i}g_i\otimes a_ib = \sum_{i,(g_i)}(g_i)_{(1)}\otimes (g_i)_{(2)}(b)a_i,$$
which is equivalent to the condition
\begin{equation}\label{comco}
\sum_{i}g_i\otimes ba_i = \sum_{i,(g_i)}(g_i)_{(1)}\otimes a_iS((g_i)_{(2)})(b),
\end{equation}
Note that this is exactly the commutation condition between $\sum_ia_i\#S(g_i)$ and $b\# 1$ in $A\#H$:
$$(b\#1)(\sum_ia_i\#S(g_i)) = \sum_iba_i\#S(g_i),$$ which corresponds to the left hand side of (\ref{comco}), coincides with 
$$(\sum_ia_i\#S(g_i))(b\#1) = \sum_{i,(S(g_i))}a_iS(g_i)_{(1)}(b)\#S(g_i)_{(2)} = \sum_{i,(g_i)}a_iS((g_i)_{(2)})(b)\#S((g_i)_{(1)}),$$ which corresponds to the right hand side of (\ref{comco}). 
\newline
The form of the Yetter-Drinfeld structure follows from the formulas in proposition \ref{rad}. Indeed, $H$-action on $R(A) = H\otimes A$ (and hence on $C_l(R(A))$) reads
$$h(\sum_ig_i\otimes a_i) = \sum_{i,(h)}h_{(1)}g_iS(h_{(3)})\otimes h_{(2)}(a_i).$$
Hence for the corresponding element of $C_{A\#H}(A)$ we have 
$$h(\sum_ia_i\#S(g_i)) =   \sum_{i,(h)}h_{(2)}(a_i)\# S(h_{(1)}g_iS(h_{(3)})) = $$ $$\sum_{(h),i}h_{(2)}(a_i)\#S^2(h_{(3)})S(g_i)S(h_{(1)}).$$
Similarly $H$-coaction for an element of $R(A) = H\otimes A$ (and hence of $C_l(R(A))$) has a form
$$\psi(\sum_ig_i\otimes a_i) = \sum_{i,(g_i)}(g_i)_{(1)}\otimes (g_i)_{(2)}\otimes a_i\in H\otimes R(A).$$
Hence for the corresponding element of $C_{A\#H}(A)$ we have 
$$\psi(\sum_ia_i\#S(g_i)) =  \sum_{i,(g_i)}(g_i)_{(1)}\otimes a_i\#S((g_i)_{(2)}) =$$ $$\sum_{i,(g_i)}S^{-1}S((g_i)_{(1)})\otimes a_i\#S((g_i)_{(2)}) =   \sum_{i,(S(g_i))}S^{-1}((S(g_i))_{(2)})\otimes a_i\#(S(g_i))_{(1)}.$$
\end{proof} 
It follows from the proof of theorem \ref{main} that the map 
$$A\#H\to H\otimes A,\quad a\#h\mapsto S^{-1}(h)\otimes a$$ induces a homomorphism (embedding) of algebras $C_{A\#H}(A)\to H\otimes A$. In particular, composing this homomorphism with $\varepsilon\otimes 1:H\otimes A\to A$ we get a homomorphism $C_{A\#H}(A)\to A$, which is the canonical homomorphism $Z(A)\to A$ from \cite{da}.

\begin{corr}\label{cor}
The centraliser $C_{A\#H}(A)$ is a commutative algebra in the braided category $\YD(H)$ of Yetter-Drinfeld modules over $H$.
\newline
Moreover, the centraliser $C_{A\#H}(A)$ is invariant with respect to Morita equivalences in $H\da Mod$, i.e. if $A$ and $B$ are two $H$-algebras and $M$ is a $B$-$A$-bimodule, equipped with compatible $H$-action, which induce an equivalence between categories of modules $M$$\otimes_A\da :{_{A}(H\da Mod)}\to{_{B}(H\da Mod)}$ then $C_{A\#H}(A)$ and $C_{B\#H}(B)$ are isomorphic as algebras in $\YD(H)$.
\end{corr}
\begin{proof}
It follows from the general properties of the full centre, established in \cite{da}.
\end{proof}
Due to Morita invariance (and more general definition given in \cite{da}, which works for module categories) it is preferable to talk about full centre of a module category rather than a full centre of an algebra. 

As examples we treat the cases of regular module category and of the module category $\Vect$ with the module structure given by the forgetful functor.
\begin{exam}\label{reg}Full centre of the regular module category.
\end{exam}
As a module category over itself $H\da Mod$ can be identified with the category of modules in $H\da Mod$ over the trivial $H$-algebra $k$. Thus by theorem \ref{main} the full centre $Z(H\da Mod)=Z(k)$ coincides with $H^{op}$ with the adjoint $H$-action 
$g(h) = \sum_{(g)}g_{(1)}hS(g_{(2)})$
and $H$-coaction given by the coproduct (which is isomorphic to the Yetter-Drinfeld structure from theorem \ref{main} by the anti-isomorphism $S$).
This can also be deduced from the fact that the functor $\cZ(H\da Mod)\to \End_{H\da Mod}(H\da Mod)$ coincides with the forgetful functor $\YD(H)\to H\da Mod$. Of course the fact that it is a quantum commutative algebra is well known (see for e.g. \cite{cvz}). 

\begin{exam}Full centre of $Vect$ as a module category.
\end{exam}
The forgetful functor $H\da Mod\to\Vect$ allows us to look at $\Vect$ as a module category over $H\da Mod$, which can be identified with the category of modules in $H\da Mod$ over the $H$-module algebra $H^*$ (see example \ref{ffmc} for a description of $H$-module algebra structure for $H^*$). Thus the full centre $Z(\Vect) = Z(H^*)$ can be identified with the centraliser $C_{H^*\# H}(H^*)$.  If $H$ is finite dimensional the answer can be simplified: in this case $C_{H^*\# H}(H^*)$ coincides with $H^*$. It can be seen in at least two ways. According to \cite{mo} the homomorphism 
$$\theta:H^*\# H\to End(H^*),\quad \theta(l\#h)(m) = lh(m),\quad h\in H, l,m\in H^*$$ is an isomorphism. Thus the centraliser $C_{H^*\# H}(H^*)$ coincides with $End_{H^*}(H^*)=(H^*)^{op}$. 
\newline
Alternatively we can use the fact that the functor $\cZ(H\da Mod)\to \End_{H\da Mod}(\Vect)$ coincides with the forgetful functor $\YD(H)\to H\da Comod$, which in the case of finite dimensional $H$ coincides with the functor $D(H)\da Mod\to H^*\da Mod$, induced by the embedding $H^*\to D(H)$. Then we can apply example \ref{reg} since $D(H) = D(H^*)$ and  the embedding $H^*\to D(H)$ coincides with the embedding $H^*\to D(H^*)$.

\section{Examples}\label{}

Here we apply the formula from theorem \ref{main} to the cases when $H$ is a group Hopf algebra $k[G]$ or its dual $k(G)$. 
Note that a $k[G]$-module algebra $A$  (a {\em $G$-algebra}) is just an (associative, unital) algebra with an action of $G$ by algebra automorphisms. Similarly, a $k(G)$-module algebra is just a {\em $G$-graded algebra}, i.e. a
$G$-graded vector space $A = \oplus_{g\in G}A_g$ with multiplication, which preserves grading
$A_fA_g\subset A_{fg}$. Indeed $A_g = p_gA$, where $p_g\in k(G)$ is a $\delta$-function on $G$ with the support at $g$. We denote by $|a|$ the degree (in $G$) of a homogeneous element $a\in A$, i.e. $a\in A_{|a|}$.

We start with recalling (for example from \cite{da}) a description of the monoidal centre $\cZ(k[G]\da Mod) = \cZ(k(G)\da Mod)$. 
We call a $G$-action on a vector space $V$ {\em compatible} with a $G$-grading $V = \oplus_{g\in G}V_g$ if $f(V_g) = V_{fgf^{-1}}$.
The following result is well-know.
\begin{prop}\label{mcgr}
The monoidal centre $\cZ(k[G]\da Mod)$ is isomorphic, as braided monoidal category, to the category $\cZ(G)$, whose objects
are $G$-graded vector spaces $X = \oplus_{g\in G}X_g$ together with a compatible $G$-action and with morphisms, which are graded
and action preserving homomorphisms of vector spaces. The tensor product in $\cZ(G)$ is the tensor product of
$G$-graded vector spaces with the $G$-action defined by
\begin{equation}\label{tp}
f(x\otimes y) = f(x)\otimes f(y),\quad x\in X, y\in Y.
\end{equation}
The monoidal unit is $I=I_e=k$ with trivial $G$-action.
\newline
The braiding is given by
\begin{equation}\label{br}
c_{X,Y}(x\otimes y) = f(y)\otimes x,\quad x\in X_f, y\in Y.
\end{equation}
\newline
For $Z\in\cZ(G)$ and $U\in\cC(G)$ the half-braiding $z_U:Z\otimes U\to U\otimes Z$ is given by
\begin{equation}\label{hb}
z_U(z\otimes u) = u\otimes g^{-1}(z),\quad u\in U_g. 
\end{equation}
\end{prop}

As an immediate application we have the following (see \cite{da0} for details).
An algebra in the category $\cZ(G)$ is a $G$-graded associative algebra $C$ together with a $G$-action such that
\begin{equation}\label{ah}
f(ab) = f(a)f(b),\quad a,b\in C.
\end{equation}
An algebra $C$ in the category $\cZ(G)$ is commutative iff
\begin{equation}\label{co}
ab = f(b)a,\quad \forall a\in C_f, b\in C.
\end{equation}

Now we are ready to describe full centres of $G$-algebras. 
\begin{prop}\label{fcga}
The full centre $Z(A)\in\cZ(G)$ of a $G$-algebra has the $G$-grading $Z(A) = \oplus_{g\in G}Z_g(A)$, where
$$Z_g(A) = \{ x\in A|\ xa=g(a)x\ \forall a\in A\}$$
with the $G$-action, induced from $A$. 
\end{prop}
\begin{proof}
An element $\sum_{g\in G}x_g\#g$of $A\#k[G]$ commutes with $a\#1$ iff for any $g\in G$ we have $ax_g = x_gg(a)$. Indeed,
$$[\sum_{g\in G}x_g\#g,a\#1] = \sum_{g\in G}(x_gg(a)-ax_g)\#g.$$ The degree of $x_g$ is $g^{-1}$. Indeed, 
$$\psi(\sum_{g\in G}x_g\#g) = \sum_{g\in G}g^{-1}\otimes x_g\#g.$$
Finally note that $ax_g = x_gg(a)$ for $x_g\in Z_{g^{-1}}$ is equivalent to $x_ga = g^{-1}(a)x_g$.
\end{proof}

Full centres of $G$-graded algebras are described by the following statement. 
\begin{prop}\label{}
Let $A$ be a $G$-graded algebra. The full centre of $A$ as an object of $\cZ(G)$ is the subspace of the space of functions $G\to A$ with homogeneous values:  
$$Z(A) = \{ z:G\to A|\ \ az(g) = z(hg)a, \forall a\in A_h\}.$$
The $G$-grading on $Z(A)$ is given by $$Z(A)_f = \{ z\in Z(A)|\ |z(g)|=g|z(e)|g^{-1} = gfg^{-1}\}.$$ The $G$-action is $g(z)(f) = z(g^{-1}f)$. 
\newline
The map $Z(A)\to A$ is the evaluation $z\mapsto z(e)$. 
\end{prop}
\begin{proof}
An element $\sum_{g\in G}z(g)\#p_g$ of $k(G)\#A$ commutes with $a\#1\in A_h\#1$ iff $az(g) = z(hg)a$ for all $f\in G$. Indeed, since for $a\in A_h$ 
$$(z\#p_g)(a\#1) = \sum_{uv=g}zp_u(a)\#p_v = za\#p_{h^{-1}g}$$
we have 
$$[\sum_{g\in G}z(g)\#p_g,a\#1] = \sum_{g\in G}(z(hg)a-az(g))\#p_g.$$
To write the $G$-action on $Z(A)$ note that since
$$\psi(\sum_{g\in G}z(g)\#p_g) = \sum_{{g,u,v\in G}\atop{uv=g}}p_{u^{-1}}\otimes z(g)\#p_v,$$
we have
$$g(\sum_{f\in G}z(f)\#p_f) = \sum_{f\in G}z(f)\#p_{gf} = \sum_{f\in G}z(g^{-1}f)\#p_f.$$
To verify the grading condition note that
$$p_g(\sum_{f\in G}z(f)\#p_f) = \sum_{{f,g_1,g_2,g_3\in G}\atop{g_1g_2g_3=g}}p_{g_2}(z(f))\#p_{g_3}p_fp_{g_1^{-1}} = \sum_{f\in G}p_{fgf^{-1}}(z(f))\# p_f.$$
Thus $\sum_{f\in G}z(f)\#p_f$ is homogeneous of degree $g$ iff $|z(f)| = fgf^{-1}$. 
\end{proof}

\begin{rem}
\end{rem}
The descriptions of full centres for $G$-graded and $G$-algebras obtained above are in complete agreement with the corresponding results of \cite{da} (section 9).

\section*{Acknowledgment}

The work on the note began during author's visit to S\~ao Paulo, where the author was accommodated by the Institute of Mathematics and Statistics of the University of S\~ao Paulo, was continued in 2009, while the author was visiting Macquarie University, and finished at Max Planck Institut F\"ur Mathematik (Bonn) in winter 2010. The author would like to thank these institutions for hospitality and excellent working conditions. The author would like to thank FAPESP (grant no. 2008/10526-1), Safety Net Grant of Macquarie University and Max Planck Gesellschaft, whose  financial support made the visits to Bonn, S\~ao Paulo and Sydney possible. The idea to write the note came to the author during the ``Colloquium on Hopf algebras, Quantum groups and Tensor categories" (La Falda, Argentina, September 2009). The author thanks Nicol\'as Andruskiewitsch for the invitation. The author also would like to thank J\"urgen Fuchs, Vyacheslav Futorny, Sonia Natale, Ingo Runkel, and York Sommerh\"auser for inspiring conversations. The author thanks Gigel Militaru and York Sommerh\"auser for pointing out to the references \cite{cm,ra,so,sz}.

\end{document}